\documentclass [a4paper] {article}

\usepackage[utf8]{inputenc}

\usepackage[french]{babel}

\usepackage{amssymb,amsmath, amsfonts}

\usepackage [dvips] {graphicx,color}

\usepackage{amsthm}

\usepackage{hyperref}

\usepackage{MnSymbol}

\newtheorem{prop} {Proposition} 
\newtheorem{lm} [prop]{Lemme} 
\newtheorem{thm} [prop] {Théorème}

\theoremstyle{definition}
\newtheorem{df}{Définition} 
\newtheorem*{df*}{Définition}

\theoremstyle{remark}
\newtheorem{rmq}{Remarque} 
\newtheorem{example}{Exemple} 
\newtheorem{exm}[example]{Exemple}

\author{Stéphane \textsc{Dugowson}
\footnote{s.dugowson@gmail.com}
}

\title {Introduction aux topos des espaces connectifs. Morita-équivalences avec les espaces topologiques et les ensembles ordonnés dans le cas fini.}
\date{4 mars 2018} 

\begin{document}

\maketitle

\paragraph{Résumé.} Ce texte comporte deux parties. Dans la première, nous rappelons et détaillons  la définition du topos de Grothendieck d'un espace connectif et des faisceaux sur un tel espace \cite{Dugowson:20161015}. Dans la seconde, nous donnons la preuve du fait que tout espace connectif fini est Morita-équivalent à un espace topologique fini, et réciproquement (nous avions présenté cette preuve dans plusieurs séminaires \cite{Dugowson:20161207, Dugowson:20170127, Dugowson:20170503}, mais nous n'en avions pas encore partagé la rédaction).

\subparagraph{\emph{Mots clés.}} Topos de Grothendieck. Espaces connectifs finis. Espaces topologiques finis. Équivalence de Morita.

\paragraph{Abstract.} \textsc{Toposes of connectivity spaces. Morita equivalences with topological spaces and partially ordered sets in the finite case ---} This paper has two parts. First, we recall and detail the definition of the Grothendieck topos of a connectivity space, that is the topos of sheaves on such a space \cite{Dugowson:20161015}. In the second part, we prove that every finite connectivity space is Morita-equivalent to a finite topological space, and vice versa (we have given this proof in several seminars \cite{Dugowson:20161207, Dugowson:20170127, Dugowson:20170503}, but we haven't yet shared this in writing).



\section{Topos associé à un espace connectif}

Dans cette section, nous reprenons et détaillons la définition du topos associé à un espace connectif donnée dans \cite{Dugowson:20161015}. 

\subsection{Rappels et précisions sur les espaces connectifs}

Pour l'essentiel, nous reprenons ici les notations et les définitions relatives aux espaces connectifs telles qu'elles figurent dans \cite{Dugowson:201012}, les  différences ou les ajouts étant précisés ci-après au fur et à mesure. En particulier :
\begin{itemize}
\item une \emph{structure connective sur un ensemble de points $X$} est un ensemble de parties de $X$ dites \emph{connexes} (pour cette structure) tel que toute famille de parties connexes d'intersection non vide a une union qui est connexe\footnote{Si $X$ est non vide, cela implique que la partie vide est  connexe, mais pour simplifier on suppose la partie vide toujours connexe, même  lorsque $X=\emptyset$.},
\item une structure connective est dite \emph{intègre} lorsque les singletons de $X$ sont connexes, 
\item pour tout ensemble $X$, on note $Cnc_X$ le treillis complet (pour l'inclusion) des structures connectives sur $X$; si $\mathcal{K}$ et $\mathcal{L}$ sont deux structures connectives sur $X$ telles que $\mathcal{K}\subset\mathcal{L}$, on dit que $\mathcal{K}$ est plus fine que $\mathcal{L}$.
\end{itemize}

\subsubsection{Les structures connectives vues comme petites catégories}

En tant qu'ensemble ordonné par l'inclusion, la structure connective $\mathcal{K}=(\mathcal{K},\subset)$ d'un espace connectif $(X,\mathcal{K})$ s'identifie à une petite catégorie. Le cas échéant, nous désignerons par $\overrightarrow{\mathcal{K}}$  l'ensemble des flèches de la catégorie $\mathcal{K}$, autrement dit l'ensemble des inclusions $(B\subset A)$ entre parties de $X$ connexes pour la structure $\mathcal{K}$.

\subsubsection{Structure connective induite sur une partie}

Pour toute partie connexe $A\in\mathcal{K}$ d'un espace connectif $(X,\mathcal{K})$, rappelons que la \emph{structure connective induite sur $A$ par $\mathcal{K}$}, que nous noterons $\mathcal{K}_{\vert A}$, est la structure la moins fine sur $A$ qui fasse de l'injection canonique $A\hookrightarrow X$ un morphisme connectif, et qu'elle est donnée par
\[\mathcal{K}_{\vert A}=\mathcal{K} \cap \mathcal{P}A.\]

\subsubsection{Structure connective engendrée par un ensemble de parties}

Étant donné $\mathcal{A}\subset\mathcal{P}X$ un ensemble de parties d'un ensemble $X$, on notera $[\mathcal{A}]$ --- plutôt que $[\mathcal{A}]_0$, comme c'était le cas dans \cite{Dugowson:201012} --- la structure connective sur $X$ engendrée par $\mathcal{A}$. C'est la structure connective  la plus fine sur $X$ qui contienne $\mathcal{A}$. 

En section \textbf{§\,\ref{subsec cribles couvrants et topo G}}, pour établir la cohérence de la définition d'une topologie de Grothendieck sur une structure connective donnée, nous aurons besoin de la proposition   \ref{prop struct engend dans B} ci-dessous selon laquelle lorsqu'un connexe contenu dans $B\subset X$ est engendré par des parties de $X$, il est en fait engendré par celles de ces parties qui sont elles-mêmes incluses dans $B$, ce qui est tout-à-fait intuitif.

\begin{prop}\label{prop struct engend dans B}
Pour tout ensemble de parties $\mathcal{C}\subset\mathcal{P}X$ et toute partie $B\subset X$ d'un ensemble $X$, on a
\[
[\mathcal{C}\cap\mathcal{P}B]=[\mathcal{C}]\cap\mathcal{P}B,\]
autrement dit $
[\mathcal{C}\cap\mathcal{P}B]=[\mathcal{C}]_{\vert B}$.
\end{prop}

\begin{proof} Notons  $\mathcal{N}B=\mathcal{P}X\setminus \mathcal{P}B$, et commençons par prouver le lemme suivant :
\begin{lm}\label{lm A}  Pour tout $\mathcal{L}\in Cnc_B$, on a $\mathcal{L}\cup \mathcal{N}B \in {Cnc}_X$ et $(\mathcal{L}\cup \mathcal{N}B)_{\vert B}=\mathcal{L}$.
\end{lm} \begin{proof}[Preuve du lemme \ref{lm A}]\label{proof lm A} Pour toute famille $(A_i)_{i\in I}$ de parties  $A_i\in \mathcal{L}\cup \mathcal{N}B\subset X$ telle que $\bigcap_{i\in I} A_i\neq\emptyset$ on a
\begin{itemize}
\item soit $A_i\in\mathcal{L}$ pour tout $i\in I$, et dans ce cas $\bigcup_{i\in I} A_i \in \mathcal{L}\subset \mathcal{L}\cup\mathcal{N}B$,
\item soit $A_i\in\mathcal{N}B$ pour au moins un $i\in I$, et dans ce cas $\bigcup_{i\in I} A_i \in \mathcal{N}B\subset \mathcal{L}\cup\mathcal{N}B$,
\end{itemize}
et donc $\mathcal{L}\cup \mathcal{N}B$ est bien une structure connective sur $X$. On a alors trivialement $(\mathcal{L}\cup \mathcal{N}B)_{\vert B}=
\{C\in \mathcal{L}\cup \mathcal{N}B, C\subset B\}
=\mathcal{L}$.
\end{proof} 
Soit maintenant 
$\mathcal{C}\subset\mathcal{P}X$ et $B\subset X$. Puisque $\mathcal{P}B$ est une structure connective sur $X$, on a $[\mathcal{P}B]=\mathcal{P}B$ et l'inclusion $[\mathcal{C}\cap \mathcal{P}B]\subset [\mathcal{C}]\cap \mathcal{P}B$ en découle immédiatement. Prouvons l'inclusion inverse. On a $\mathcal{C}=(\mathcal{C}\cap \mathcal{P}B)\cup(\mathcal{C}\cap\mathcal{N}B)$, d'où $\mathcal{C}\subset(\mathcal{C}\cap \mathcal{P}B)\cup\mathcal{N}B\subset [\mathcal{C}\cap \mathcal{P}B]\cup\mathcal{N}B$ d'où, d'après le lemme \ref{lm A} appliqué à $\mathcal{L}=[\mathcal{C}\cap \mathcal{P}B]$, $[\mathcal{C}]\subset [\mathcal{C}\cap \mathcal{P}B]\cup\mathcal{N}B$, d'où l'on déduit finalement l'inclusion annoncée.
\end{proof}

\begin{rmq}\label{rmq struct integ engend dans B} La structure connective \emph{intègre} engendrée par $\mathcal{C}$ --- autrement dit la structure connective intègre la plus fine contenant  $\mathcal{C}$ --- est notée  $[\mathcal{C}]_1$, 
et on a $[\mathcal{C}]_1=[\mathcal{C}] \cup (\bigcup_{x\in X} \{\{x\}\})$. On déduit immédiatement de la proposition \ref{prop struct engend dans B} ci-dessus que pour tout $\mathcal{C}\subset\mathcal{P}X$ et tout $B\subset X$, 
\[
[\mathcal{C}\cap\mathcal{P}B]_1\cap\mathcal{P}B=[\mathcal{C}]_1\cap\mathcal{P}B.\]
\end{rmq}

\subsubsection{Connexes irréductibles}

Rappelons qu'une partie connexe $A\in\mathcal{K}$ de $X$ est dite \emph{irréductible} si et seulement si $A\notin [\mathcal{K}\setminus \{A\}]$. Remarquons que le fait, dans cette définition, de faire appel à l'engendrement de structures connectives \emph{non nécessairement intègres} fait que, lorsqu'ils sont connexes, les points\footnote{C'est-à-dire les singletons de $X$.} sont irréductibles. Par contre, la partie vide de $X$ n'est jamais irréductible.\\

On notera $\mathrm{Irr}_{(X,\mathcal{K})}$ l'ensemble des connexes irréductibles de l'espace ${(X,\mathcal{K})}$.

\subsection{Cribles dans une structure connective}

La catégorie $\mathcal{K}=(\mathcal{K},\subset)$ étant un ensemble ordonné, un crible $\sigma$ dans cette catégorie  pourra s'écrire 
\[\sigma=(\dot{\sigma}, A),\] où 
\begin{itemize}
\item $A\in \mathcal{K}$, codomaine commun des flèches $(B\subset A)\in\sigma$, sera appelé la \emph{cible} ou le \emph{codomaine} du crible $\sigma$, et pourra être noté $A=\mathrm{cod}(\sigma)$,
\item $\dot{\sigma}\subset \mathcal{K}_{\vert A}$, que nous appellerons le domaine du crible, que nous désignerons par $\dot{\sigma}=\mathrm{dom}(\sigma)$  et dont les éléments seront simplement appelés les \emph{éléments du crible}, devra vérifier
 \[
\forall B\in\dot{\sigma}, \mathcal{K}\cap \mathcal{P}B\subset\dot{\sigma}.
\]
\end{itemize}

On notera  $\mathrm{SV}_\mathcal{K}$ l'ensemble
des cribles dans $\mathcal{K}$ et, pour tout connexe $A\in\mathcal{K}$,  $\mathrm{SV}_\mathcal{K}(A)$ --- ou simplement $\mathrm{SV}(A)$ s'il n'y a pas d'ambiguïté ---  l'ensemble des cribles sur $A$ dans la catégorie $\mathcal{K}$.

\begin{rmq}
La \emph{catégorie} $\mathbf{SV}_\mathcal{K}$  des cribles dans  $\mathcal{K}$ est elle-même un ensemble ordonné, une flèche $\varphi: \beta=(\dot{\beta}, B)\rightarrow\alpha=(\dot{\alpha}, A)$ exprimant, si elle a lieu, l'inclusion $\beta\subset \alpha$ définie par les deux conditions $B\subset A$ \emph{et} $\dot{\beta}\subset \dot{\alpha}$.
\end{rmq}

\begin{rmq} 
Il est important de distinguer, pour un connexe $A\in\mathcal{K}$, le crible vide $(\emptyset,A)$, qui n'a pas d'élément, du crible non vide minimal, $(\{\emptyset\}, A)$, qui a un unique élément, à savoir la partie vide de $X$, qui est toujours connexe. Cette remarque vaut en particulier pour la partie vide $A=\emptyset$, qui est donc la cible de deux cribles distincts.
\end{rmq}

Pour tout connexe $A\in\mathcal{K}$, le \emph{crible maximal sur $A$ dans $\mathcal{K}$}, que nous noterons $\mathrm{msv}_\mathcal{K}(A)$  est 
\[
\mathrm{msv}_\mathcal{K}(A)
=(\mathcal{K}\cap{\mathcal{P}A}, A)
=(\mathcal{K}_{\vert A}, A).
\]

La proposition suivante découle immédiatement de la définition d'un crible.
\begin{prop}\label{prop crible max sur A} Pour tout crible $\sigma$  sur $A$ dans $\mathcal{K}$, on a l'équivalence
\[\sigma=\mathrm{msv}_\mathcal{K}(A) \Leftrightarrow A\in\dot{\sigma}.\]
\end{prop}

\subsubsection{Image réciproque d'un crible}

Étant donné $\sigma=(\dot{\sigma},A)\in\mathrm{Cr}_\mathcal{K}(A)$ un crible sur $A\in\mathcal{K}$, et $B\in\mathcal{K}$ un connexe quelconque inclus dans $A$, la définition catégorique générale de l'image réciproque par une flèche d'un crible de même cible que cette flèche s'écrit ici
\[
\varphi^*(\sigma)=(\dot{\sigma}\cap \mathcal{P}B,B),
\]
où $\varphi=(B\subset A)\in \overrightarrow{\mathcal{K}}$ est la flèche de la catégorie $\mathcal{K}$ exprimant l'inclusion $A\subset B$.
La formule ci-dessus justifie que $\varphi^*(\sigma)$ soit également noté $\sigma_{\vert B}$ et appelé \emph{la restriction de $\sigma$ à $B\subset A$}, de sorte que $\sigma_{\vert B}=(\dot{\sigma}\cap \mathcal{P}B,B)$ ou encore, posant $\dot{\sigma}_{\vert B}=\dot{\sigma}\cap \mathcal{P}B$
\[
\sigma_{\vert B}=(\dot{\sigma}_{\vert B},B).
\]

\begin{rmq} Pour tout crible $\sigma$ sur $A$ dans $\mathcal{K}$, et pour tout $B\subset A$, on a trivialement l'équivalence suivante, qui généralise la proposition \ref{prop crible max sur A} :
\[\sigma_{\vert B}=\mathrm{msv}_\mathcal{K}(B) 
\Leftrightarrow 
B\in\dot{\sigma}.\]
\end{rmq}

\subsection{Cribles couvrants, topologie de Grothendieck sur $\mathcal{K}$}\label{subsec cribles couvrants et topo G}

\subsubsection{Définition des cribles couvrants}

Pour tout connexe $A\in\mathcal{K}$, on pose
\[
J_{(X,\mathcal{K})}(A)=J(A)=\{
\sigma\in \mathrm{SV}(A), 
[\dot{\sigma}] 
\supset \mathrm{dom}(\mathrm{msv}_\mathcal{K}(A))
\},
\]
où, rappelons-le, $[\dot{\sigma}]$ désigne la structure connective \emph{non nécessairement intègre} engendrée par l'ensemble $\dot{\sigma}\subset \mathcal{K}$. De façon équivalente, on a
\[
J(A)=\{
\sigma\in \mathrm{SV}(A), 
[\dot{\sigma}] 
= \mathcal{K}_{\vert A}
\}.
\]

Anticipant le théorème \ref{thm topo G sur K} donné ci-dessous en section \textbf{§\,\ref{subsubsec thm site et topos}} qui justifiera réellement cette terminologie, les éléments de $J(A)$ seront dès à présent appelés les \emph{cribles couvrants $A$}, ou les \emph{cribles qui recouvrent $A$}.

\subsubsection{Cas des connexes irréductibles}

\begin{prop}
Un connexe $K\in\mathcal{K}$ est irréductible si et seulement si $J(K)$ est un singleton, seul le crible maximal étant alors couvrant :
\[K \in\mathrm{Irr}_{(X,\mathcal{K})}\Leftrightarrow J(K)=\{\mathrm{msv}_{\mathcal{K}}(K)\}.\] 
\end{prop} 
\begin{proof}
Si $K$ est irréductible, seul un crible dont le domaine contient $K$ peut engendrer $K$ lui-même, ce qui caractérise le crible maximal sur $K$. Et réciproquement, si $J(K)$ ne contient que le crible maximal, alors en particulier le crible $(\mathrm{msv}_\mathcal{K}(K)\setminus \{K\},K)$ n'est pas couvrant, ce qui implique que $K\notin [\mathrm{msv}_\mathcal{K}(K)\setminus \{K\}]$, autrement dit que $K$ est irréductible.
\end{proof}

Par exemple, la partie vide $\emptyset\subset X$ n'étant pas irréductible, les deux cribles sur $\emptyset$ sont nécessairement couvrants.

\subsubsection{Exemples (et un contre-exemple) de cribles couvrants}

\begin{exm}
[Un espace fini d'ordre $2$]
On prend sur $X=\{a,b,c,d,e\}$ la structure connective \[\mathcal{K}=\{\emptyset,\{a\}, \{b\},\{c\},\{d\},\{e\}, \{a,b\}, \{b,c,d\}, \{a,b,c,d\}, X\}.\] Tous les $J(K)$ sont réduits au crible maximal sur $K$, à l'exception du vide $\emptyset$ et de $\{a,b,c,d\}$ qui est également couvert par le crible \[\{\emptyset,\{a\}, \{b\},\{c\},\{d\}, \{a,b\}, \{b,c,d\}\}.\]
\end{exm}

\begin{exm}
[Un crible couvrant la droite réelle]
 Soit $\mathcal{K}_\mathbf{R}$ la structure connective usuelle sur $\mathbf{R}$, c'est-à-dire l'ensemble des intervalles. Pour tout intervalle $I$ et tout $\epsilon> 0$,  l'ensemble des sous-intervalles de $I$ de longueur $<\epsilon$ est un crible couvrant. 
\end{exm}

\begin{exm}[Contre-exemple]
Contrairement à ce qui se passe pour les cribles couvrants un ouvert dans une topologie, il ne suffit pas, pour qu'un crible $\sigma\in\mathrm{Cr}_\mathcal{K}(A)$ soit couvrant que l'union des $B\in \dot{\sigma}$ soit égale à $A$. Par exemple, considérons un ensemble de cinq points $X=\{x_1, x_2, x_3, x_4, x_5\}$ muni de la structure connective intègre $\mathcal{K}$ engendrée par 
$
\{x_1, x_2, x_3\}$,
$\{x_2, x_3, x_4\}$ et 
$\{x_3, x_4, x_5\}$.
 Le crible $\sigma$ sur le connexe $X\in\mathcal{K}$  donné par
\[
\dot{\sigma}=
\{
\emptyset,
\{x_1\}, 
\{x_2\},
\{x_3\},
\{x_4\},
\{x_5\},
\{x_1, x_2, x_3\},
\{x_3, x_4, x_5\}
\}
\] vérifie $X\in [\dot{\sigma}]$, mais n'est pas couvrant pour autant puisque par exemple $\{x_2, x_3, x_4\}\notin [\dot{\sigma}]$.
\end{exm}

\subsubsection{Site et topos associés à un espace connectif}
\label{subsubsec thm site et topos}
\begin{thm} \label{thm topo G sur K} L'application
$J=J_{(X,\mathcal{K})}$ constitue une topologie de Grothendieck sur la catégorie
$\mathcal{K}$. 
\end{thm}
\begin{proof} Nous devons vérifier que, pour tout connexe $A\in\mathcal{K}$,
\begin{enumerate}
\item le crible maximal sur $A$ est couvrant,
\item pour tout connexe $B\subset A$ et tout $\sigma\in J(A)$, on a $\sigma_{\vert B}\in J(B)$,
\item pour tout $\mu\in \mathrm{SV}_\mathcal{K}(A)$, s'il existe $\sigma\in J(A)$ tel que pour tout $B\in\dot{\sigma}$ on ait $\mu_{\vert B}\in J(B)$, alors $\mu\in J(A)$.
\end{enumerate}
Par définition de $J(A)$, la première de ces trois propriété est trivialement satisfaite. Pour vérifier la seconde, donnons-nous une partie connexe $B\subset A$ et un crible couvrant $\sigma\in J(A)$. On a, d'après la proposition \ref{prop struct engend dans B} 
\[
[dom(\sigma_{\vert B})]=[\dot{\sigma}\cap\mathcal{P}B]=[\dot{\sigma}]\cap\mathcal{P}B,
\] mais puisque $\sigma\in J(A)$, on a $[\dot{\sigma}]=\mathcal{K}_{\vert A}$, d'où $[dom(\sigma_{\vert B})]=\mathcal{K}_{\vert B}$, autrement dit $\sigma_{\vert B}\in J(B)$. Vérifions enfin la troisième propriété. Soit donc $\mu\in \mathrm{SV}_\mathcal{K}(A)$ tel qu'il existe $\sigma\in J(A)$ tel que pour tout $B\in\dot{\sigma}$ on ait $\mu_{\vert B}\in J(B)$. Par hypothèse, on a donc $[\dot{\sigma}]\supset \mathcal{K}_{\vert A}$ et, pour tout $B\in\dot{\sigma}$, 
$
[\dot{\mu}]\supset
[\dot{\mu}\cap\mathcal{P}B]\supset\mathcal{K}_{\vert B}\ni B,
$ d'où 
$
[\dot{\mu}]\supset \dot{\sigma}
$ de sorte que 
$
[\dot{\mu}]\supset [\dot{\sigma}]\supset \mathcal{K}_{\vert A},
$ autrement dit $\mu\in J(A)$.
\end{proof}

\begin{df}[Topos associé à un espace connectif] Étant donné $(X,\mathcal{K})$ un espace connectif, on appelle \emph{site associé à cet espace} le site $(\mathcal{K},J)$, où pour tout $A\in\mathcal{K}$ l'ensemble de cribles $J(A)\subset \mathrm{SV}_\mathcal{K}(A) $  est défini par
\[
J(A)=\{
\sigma\in \mathrm{SV}_\mathcal{K}(A), 
[\dot{\sigma}] 
= \mathcal{K}_{\vert A}
\}.
\]
Le \emph{topos $\mathcal{TG}_{(X,\mathcal{K})}=Sh(X,\mathcal{K})$ 
associé à l'espace connectif 
$(X,\mathcal{K})$} est le topos des faisceaux d'ensembles sur le site
 $(\mathcal{K},J)$.
\end{df}

\subsection{Faisceaux sur un espace connectif $(X,\mathcal{K})$}

Dans cette section, nous regardons comment la définition générale d'un faisceau sur un site se décline dans le cas particulier du site défini par un espace connectif.

\subsubsection{\og Foncteurs contravariants\fg, préfaisceaux sur $\mathcal{K}$}

\begin{rmq}
En général, l'expression en toutes lettres \og \emph{$F$ est un foncteur contravariant de $\mathbf{C}$ dans $\mathbf{E}$}\fg\, signifie que $F$ est un foncteur (covariant) $F:\mathbf{C}^{op}\rightarrow \mathbf{E}$, de sorte qu'un tel foncteur peut de façon équivalente être désigné comme \og \emph{foncteur contravariant $F:\mathbf{C}^{op}\rightarrow \mathbf{E}$}\fg\, ou comme \og \emph{foncteur $F:\mathbf{C}^{op}\rightarrow \mathbf{E}$}\fg.
\end{rmq}

Un \emph{préfaisceau (d'ensembles) $F$ sur un espace connectif $(X,\mathcal{K})$} est un foncteur contravariant  $F:\mathcal{K}^{op}\rightarrow\mathbf{Ens}$, autrement dit c'est un foncteur $(\mathcal{K},\supset)\rightarrow\mathbf{Ens}$.

\subsubsection{Limite projective d'un préfaisceau sur un ensemble de connexes}

Étant donné $(X,\mathcal{K})$ un espace connectif, 
$F:\mathcal{K}^{op}\rightarrow\mathbf{Ens}$ un préfaisceau et
 $\dot{\mathcal{A}}\subset\mathcal{K}$ un ensemble de $\mathcal{K}$-connexes, 
 on appelle \emph{limite projective de $F$ sur $\dot{\mathcal{A}}$}, 
 ou encore, désignant par $\mathcal{A}=(\dot{\mathcal{A}},\subset)$ la sous-catégorie pleine de $\mathcal{K}$ définie par $\dot{\mathcal{A}}$, 
 \emph{limite projective de $F$ sur $\mathcal{A}$} 
 et l'on note ${\lim_{\leftarrow}}_{\mathcal{A}} F$
 ou ${\lim_{\leftarrow}}_{B\in\mathcal{A}} FB$
 la réalisation habituelle de la limite projective du diagramme d'ensembles définie par le foncteur covariant $F$ sur la catégorie $\mathcal{A}^{op}$, à savoir
\[
{\lim_{\longleftarrow}}_{\mathcal{A}} F
= {\lim_{\longleftarrow}}_{B\in\mathcal{A}} FB
=
\{
(f_B)_{B\in\dot{\mathcal{A}}}
\in\prod_{B\in\dot{\mathcal{A}}}FB,
\forall (B_1\supset C, B_2\supset C)\in \overrightarrow{\mathcal{A}^{op}}, 
{f_{B_1}}_{\vert C}={f_{B_2}}_{\vert C}
\}
\]
où le résultat $\rho_{BC}(f_B)$ de l'application de restriction $\rho_{BC}=F(B\supset C)$ à un élément $f_B\in FB$ est notée selon l'usage habituel ${f_B}_{\vert C}$.

\begin{rmq}\label{rmq * singleton produit}
Dès lors que la notion de produit d'une famille d'ensembles est considérée comme définie sans ambiguïté, la définition ci-dessus est elle-même dépourvue d'ambiguïté. En particulier, on suppose que le singleton défini par produit de la famille vide d'ensemble est connu, et l'on notera $*$ son unique élément. 
\end{rmq}

\subsubsection{Limite projective d'un préfaisceau sur un crible de connexes}

$(X,\mathcal{K})$ désignant toujours un espace connectif et 
$F:\mathcal{K}^{op}\rightarrow\mathbf{Ens}$ un préfaisceau, étant donné $A\in \mathcal{K}$ et $\sigma=(\dot{\sigma}, A)\in\mathrm{Cr}_\mathcal{K}(A)$, on appelle \emph{limite projective de $F$ sur $\sigma$}, et l'on note  ${\lim_{\leftarrow}}_{\sigma} F$
ou ${\lim_{\leftarrow}}_{\dot{\sigma}} F$ 
ou
${\lim_{\leftarrow}}_{(B\subset A)\in \sigma} FB$
ou encore ${\lim_{\leftarrow}}_{B\in\dot{\sigma}} FB$ la limite projective de $F$ sur l'ensemble $\dot{\mathcal{A}}=\dot{\sigma}$.

\subsubsection{Relation \og$FA\simeq {\lim_{\leftarrow}}_{\sigma} F$\fg}

$(X,\mathcal{K})$ désignant toujours un espace connectif et 
$F:\mathcal{K}^{op}\rightarrow\mathbf{Ens}$ un préfaisceau,  $A\in\mathcal{K}$ un connexe et $\sigma$ un crible sur $A$, on écrit $FA\simeq {\lim_{\leftarrow}}_{\sigma} F$ pour exprimer précisément le fait suivant : l'application 
\[\theta :FA\rightarrow {\lim_{\leftarrow}}_{\sigma} F\] définie pour tout $f_A\in FA$ par $\theta(f)=({f_A}_{\vert B})_{B\in\dot{\sigma}}$ est une bijection. De façon équivalente, on a donc $FA\simeq {\lim_{\leftarrow}}_{\sigma} F$ si et seulement s'il existe une application \[\lambda:{\lim_{\leftarrow}}_{\sigma} F\rightarrow  FA\] réciproque de $\theta$, autrement une application $\lambda$ qui \og recolle\fg\, toute famille \og cohérente\fg\, $(f_B)_{B\in\dot{\sigma}}\in {\lim_{\leftarrow}}_{\sigma} F $ en un unique élément $f=\lambda((f_B)_{B\in\dot{\sigma}})\in FA$ tel que $f_B=f_{\vert B}$ pour tout $B\in\dot{\sigma}$.

\subsubsection{Faisceaux sur $(X,\mathcal{K})$}

Un faisceau $F\in Sh(X,\mathcal{K})$ est un préfaisceau $F:\mathcal{K}^{op}\rightarrow \mathbf{Ens}$ tel pour tout connexe $A\in\mathcal{K}$ et tout crible couvrant $\sigma\in J(A)$, on a $FA\simeq {\lim_{\leftarrow}}_{\sigma} F$.

\begin{prop}
Pour tout faisceau $F\in Sh(X,\mathcal{K})$, on a $F(\emptyset)=\{*\}$, où $*$ est l'unique élément du produit de la famille vide d'ensemble (voir la remarque \ref{rmq * singleton produit} ci-dessus).
\end{prop}
\begin{proof}
C'est une conséquence immédiate du fait que $\emptyset\in\mathcal{K}$ admet comme crible couvrant le crible vide.
\end{proof}

\subsubsection{Remarque : $J$ est sous-canonique}
\label{subsubs rmq J sous-canonique mais non canonique}

La topologie de Grothendieck $J$ sur $(\mathcal{K},\subset)$ associée à un espace connectif $(X,\mathcal{K})$ est \emph{sous-canonique}, ce qui signifie que tout préfaisceau représentable est un faisceau. 
En effet, pour tout connexe $C\in\mathcal{K}$, considérons le préfaisceau $F_C$ défini pour tout connexe $A\in\mathcal{K}$ par $F_C(A)=\{*\}$ si $A\subset C$ et $F_C(A)=\emptyset$ si $A \not\subset C$. Dans le premier cas, on a $F(A)=\lim_{\leftarrow \sigma} F$ pour tout crible $\sigma$ sur $A$, et en particulier pour tout crible couvrant, puisque $B\in\dot{\sigma}\Rightarrow B\subset C$. 
Dans le cas où $A \not\subset C$, si $\sigma\in J(A)$ alors il y a un $B\in\dot{\sigma}$ tel que $B \not\subset C$ --- si ce n'était pas le cas, on aurait $\dot{\sigma}\subset \mathcal{P}C$, d'où $A\in [\dot{\sigma}]\subset \mathcal{P}C$, ce qui est absurde --- de sorte que $F(B)=\emptyset$ et $F(A)=\lim_{\leftarrow \sigma} F$. 

Par contre, en général, $J$ n'est pas la topologie canonique sur l'ensemble ordonné $(\mathcal{K},\subset)$. 
Par exemple, dans l'ensemble ordonné  $\mathcal{K}=\{\emptyset, \{x_1\}, \{x_2\}, X\}$ des connexes de l'espace constitué de deux points connexes connectés\footnote{Voir l'exemple \ref{exm deux points connexes connectes}  page \pageref{exm deux points connexes connectes}.}, le maximum $X=\{x_1, x_2\}$ est irréductible de sorte que seul le crible maximal couvre $X$, tandis que ce même ensemble ordonné, pouvant être vu comme celui des ouverts d'un espace topologique constitué de deux points ouverts, a pour topologie canonique celle de l'espace topologique en question pour laquelle $X$ est couvert par $\{\emptyset, \{x_1\}, \{x_2\}\}$.

\subsection{Exemples de topos d'espaces connectifs finis}

\begin{exm}[topos vide]
Nous appellerons \emph{topos vide} le topos associé à l'espace connectif vide, d'ensemble de points $X=\emptyset$ et de structure connective le singleton $\mathcal{K}=\{\emptyset\}$. 
On a $J(\emptyset)=\{\emptyset, \{\emptyset\}\}$. Un faisceau $F$ sur l'espace vide doit vérifier $F(\emptyset)=lim_\emptyset$, la limite projective de la famille vide d'ensemble, de sorte que $F(\emptyset)=\{*\}$. 
Il y a donc un seul faisceau sur le vide, identifié au singleton $\{*\}$ dont l'identité est donc le seul morphisme. 
Le topos $\mathcal{TG}_{(\emptyset,\{\emptyset\})}$ est ainsi le topos dégénéré, réduit à la catégorie $\mathbf{1}$.
\end{exm}

\begin{exm}[Un point non connexe] Le topos de l'espace connectif constitué d'un seul point non connexe est le topos vide.
\end{exm}

\begin{exm}[Un point connexe] Le topos de l'espace connectif constitué d'un seul point connexe est équivalent au topos des ensembles.
\end{exm}

\begin{exm}[Deux points non connexes connectés] L'espace connectif constitué de deux points non connexes connectés est trivialement Morita-équivalent à un seul point connexe.
\end{exm}

\begin{exm}[Un point connexe et un point non connexe connectés] L'espace connectif constitué de deux points connectés dont l'un est connexe et l'autre non possède deux connexes irréductibles emboîtés, de sorte que son topos de Grothendieck est le topos des applications.
\end{exm}

\begin{exm}[Deux points connexes connectés]\label{exm deux points connexes connectes}
 $X=\{x_1, x_2\}$, $\mathcal{K}=\{\emptyset, \{x_1\},\{x_2\}, X\}$. \`{A} part la partie vide, tous les connexes sont irréductibles. En particulier, le crible $\sigma$ sur $X$ défini par $\dot{\sigma}=\{\emptyset,  \{x_1\},\{x_2\}\}$ ne couvre pas $X$, seul le crible maximal sur $X$ le couvrant. Les faisceaux sur cet espace sont les couples d'applications de même domaine, de la forme $(A\stackrel{f_1}{\rightarrow}B_1, A\stackrel{f_2}{\rightarrow}B_2)$.
\end{exm}

\begin{exm}[Topos borroméen]\label{exm borro}
$X=\{x_1, x_2, x_3\}$ et $\mathcal{K}=\{\emptyset, \{x_1\}, \{x_2\}, \{x_3\}, X\}$.
Les faisceaux sur l'espace borroméen sont les triplets d'applications de même domaine, de la forme $(A\stackrel{f_1}{\rightarrow}B_1, A\stackrel{f_2}{\rightarrow}B_2, A\stackrel{f_3}{\rightarrow}B_3)$.

\end{exm}

\section{Morita équivalences entre espaces connectifs finis et espaces topologiques finis}\label{sec Morita equi}

Conformément à la notion toposique de Morita-équivalence, s'ils ont des topos équivalents, deux espaces connectifs, ou encore un espace topologique et un espace connectif, etc., seront dit Morita-équivalents. Nous écrirons $A\sim_\mathrm{M} B$ pour exprimer que les objets mathématiques $A$ et $B$ sont Morita-équivalents, autrement dit qu'ils déterminent des sites qui, s'ils sont encore notés $A$ et $B$, vérifient $Sh(A)\simeq Sh(B)$.

\begin{prop}\label{prop tout connectif fini est Morita topologique}
Pour tout espace connectif fini, il existe à \emph{iso} près un unique espace topologique fini sobre qui lui soit Morita-équivalent.
\end{prop}

Nous donnons en appendice (section \ref{append autre preuve}) une courte preuve de la proposition \ref{prop tout connectif fini est Morita topologique} ci-dessus, preuve qui s'appuie sur des théorèmes connus mais qui présente l'inconvénient de ne pas préciser comment, concrètement, obtenir un espace topologique Morita-équivalent à un espace connectif fini donné.

Les considérations qui suivent donnent au contraire un procédé explicite de construction d'un tel espace topologique, et montrent également comment associer à tout espace topologique  un espace connectif qui lui est Morita-équivalent. Ces équivalences sont établies via les ensembles ordonnés, le topos d'un ensemble ordonné étant le topos de ses préfaisceaux (autrement dit, un ensemble ordonné définit un site en prenant pour topologie de Grothendieck la topologie triviale, pour laquelle l'unique crible couvrant un élément donné quelconque est le crible maximal).

Plus précisément, notant $\mathbf{FCnc}$ la catégorie des espaces connectifs finis, $\mathbf{FPos}$ celle des ensembles (partiellement) ordonnés finis et $\mathbf{FTop}$ celle des espaces topologiques finis, on se propose  dans les sections suivantes de définir quatre applications 
\[\mathcal{G}:\vert \mathbf{FCnc}\vert \rightarrow \vert \mathbf{FPos}\vert,\] 
\[\mathcal{Z}:\vert \mathbf{FPos}\vert \rightarrow \vert \mathbf{FCnc}\vert,\]  
\[\mathcal{H}:\vert \mathbf{FTop}\vert \rightarrow \vert \mathbf{FPos}\vert,\] 
\[\mathcal{E}:\vert \mathbf{FPos}\vert \rightarrow \vert \mathbf{FTop}\vert\]
dont nous montrerons qu'elles vérifient les propriétés suivantes :
\begin{itemize}
\item pour tout espace connectif fini $(X,\mathcal{K})$, on a 
\[(X,\mathcal{K})\sim_{\mathrm{M}}\mathcal{G}_{(X,\mathcal{K})}\in \vert \mathbf{FPos}\vert ,\]
\item tout espace topologique fini $(E,\mathcal{T})$, on a
 \[(E,\mathcal{T})\sim_{\mathrm{M}}\mathcal{H}_{(E,\mathcal{T})}\in\vert \mathbf{FPos}\vert,\]
\item pour tout ensemble partiellement ordonné fini $(\mathcal{A},\leq)$, on a
 \[(\mathcal{A},\leq)\sim_{\mathrm{M}} \mathcal{Z}_{(\mathcal{A},\leq)}\in \vert \mathbf{FCnc}\vert\] et  \[(\mathcal{A},\leq)\sim_{\mathrm{M}}\mathcal{E}_{(\mathcal{A},\leq)}\in \vert \mathbf{FTop}\vert.\]
\end{itemize}

\mbox{}

Il est parfaitement possible d'établir ces équivalences en explicitant les morphismes géo\-mé\-triques correspondants entre les topos de faisceaux concernés. \`{A} titre indicatif, je reproduis en appendice (section \ref{subs appendice verif de G par les faisceaux}) une telle vérification pour l'équivalence $(X,\mathcal{K})\sim_{\mathrm{M}}\mathcal{G}_{(X,\mathcal{K})}\in \vert \mathbf{FPos}\vert$. Mais comme me l'a fait remarquer Olivia Caramello, que je remercie ici pour cette observation, ces équivalences découlent immédiatement du lemme de comparaison de Grothendieck, que nous commencerons donc par rappeler dans la section suivante.

\subsection{Lemme de comparaison de Grothendieck}

Rappelons que,
un site $(\mathbf{C},J)$ étant donné, le lemme de comparaison de Grothendieck\footnote{\cite{Grothendieck:1972_SGA4} (SGA4), \textbf{§\,III.4.1}.} s'applique en particulier au cas d'une sous-catégorie pleine $\mathbf{D}$ de $\mathbf{C}$ \emph{dont les objets recouvrent ceux de $\mathbf{C}$} --- on dit que $\mathbf{D}$ est \emph{$J$-dense} --- au sens où, selon la formulation de ce lemme donnée par Olivia Caramello dans \cite{Caramello:20140923} : pour tout objet $c\in\dot{\mathbf{C}}$,
 le crible sur $c$ engendré par les flèches de la forme $d\rightarrow c$ avec $d\in\mathbf{D}$ couvre $c$. 
 
Lorsque cette condition est satisfaite, on a l'équivalence \[Sh(\mathbf{C},J)\simeq Sh(\mathbf{D}, J_{\vert \mathbf{D}}),\] où $J_{\vert \mathbf{D}}$ est la topologie de Grothendieck induite sur $\mathbf{D}$ par $J$, c'est-à-dire celle pour laquelle les cribles couvrants  $d\in\dot{\mathbf{D}}$ sont ceux qui engendrent dans $\mathbf{C}$ des cribles couvrant  $d\in\dot{\mathbf{C}}$.

\subsection{Application $\mathcal{G}$ telle que $(X,\mathcal{K}) \sim_{\mathrm{M}}\mathcal{G}_{(X,\mathcal{K})}\in \vert \mathbf{FPos}\vert$}
\label{subs equi G}

On sait\footnote{Selon la proposition 5 de \cite{Dugowson:201012}, qui résulte d'une récurrence finie évidente.} que les connexes irréductibles d'un espace connectif \emph{fini} engendrent sa structure\footnote{Dans le cas d'un espace connectif infini, ce n'est plus vrai en général. Par exemple, les seuls connexes irréductibles de la droite connective usuelle $\mathbf{R}$ sont les points, qui n'engendrent aucun autre connexe non vide.}. Par définition des cribles couvrants sur un connexe, on en déduit que les cribles couvrant un connexe $A\in\mathcal{K}$ d'un espace connectif fini sont exactement ceux qui contiennent tous les irréductibles inclus dans $A$.

Un espace connectif fini quelconque (non nécessairement intègre) $(X,\mathcal{K})$ étant donné, notons  $(\mathcal{G},\leq)$ l'ensemble ordonné par l'inclusion des parties connexes irréductibles\footnote{Il s'agit donc du graphe générique de l'espace connectif, tel qu'il a été défini dans \cite{Dugowson:201306}. Dans le cas particulier d'un espace fini intègre, c'est le graphe défini par les relations de contenance immédiate entre irréductibles --- \emph{i. e.} il y a une arrête $L\rightarrow K$ si et seulement si $L\supset K$ mais il n'existe par d'irréductible intermédiaire $M$ tel que $L\supset M\supset K$ ---  qui, dans \cite{Dugowson:201012}, avait été appelé le \emph{graphe générique} de cet espace.} : 
\[(\mathcal{G},\leq)= (\mathrm{Irr}_{(X,\mathcal{K})},\subset).\]

L'équivalence annoncée $Sh(X,\mathcal{K})\simeq \widehat{(\mathcal{G},\leq)}$ résulte alors immédiatement du lemme de comparaison de Grothendieck. 
En effet, la condition de densité est trivialement satisfaite avec $\mathbf{C}=(\mathcal{K},\subset)$ muni de la topologie $J=J_{(X,\mathcal{K})}$ et $\mathbf{D}={(\mathrm{Irr}_{(X,\mathcal{K})},\subset)}$, puisqu'elle revient précisément à dire que, pour toute partie connexe $C\in\mathcal{K}$, le crible sur $C$ engendrée par les parties connexes irréductibles $K\subset C$  est couvrant, ce qui est bien le cas\footnote{La topologie $J$ étant sous-canonique d'après la remarque de la section \textbf{§\,\ref{subsubs rmq J sous-canonique mais non canonique}}, et bien que cette condition ne soit en réalité pas  nécessaire ici pour conclure, on peut aussi appliquer la version du lemme de comparaison donnée en appendice, \textbf{§\,4, corrolaire 3} de \cite{MacLaneMoerdijk:1992}, en prenant encore comme crible couvrant $C$ le crible engendré par les parties connexes irréductibles incluses dans $C$.}. Un crible sur $K\in(\mathcal{G},\leq)$ étant alors couvrant pour la topologie induite $J_{\vert \mathbf{D}}$ si et seulement s'il contient tous les irréductibles inclus dans $K$, donc aussi $K$ lui-même, autrement dit si et seulement s'il s'agit du crible maximal sur $K$, on en déduit que $J_{\vert \mathbf{D}}$ est la topologie triviale sur ${(\mathcal{G},\leq)}$, d'où finalement
\[
Sh(X,\mathcal{K})\simeq \widehat{(\mathcal{G},\leq)}.
\]

On a donc prouvé la proposition suivante.

\begin{prop}\label{prop equi FCnc-FPos}
Pour tout espace connectif fini $(X, mathcal{K})$, on a 
\[(X, mathcal{K})\sim_{\mathrm{M}}(\mathrm{Irr}_{(X,\mathcal{K})},\subset).\]
\end{prop}

\begin{rmq}[$\mathcal{G}$ n'est pas fonctoriel]\label{rmq G non fonctoriel}

 Remarquons  que l'application $\mathcal{G}:\vert{\mathbf{FCnc}}\vert\rightarrow \vert{\mathbf{FPos}}\vert$ qui à tout espace connectif fini $(X,\mathcal{K})$ associe $\mathcal{G}_{(X,\mathcal{K})}=(\mathrm{Irr}_{(X,\mathcal{K})},\subset)$ ne se prolonge pas en un foncteur $\mathbf{FCnc}\rightarrow \mathbf{FPos}$.

Pour s'en convaincre, il devrait suffire de considérer le morphisme connectif $f=Id_X$ de l'espace borroméen $(X=\{x_1,x_2,x_3\}, \mathcal{K}=\{\emptyset,\{x_1\}, \{x_2\}, \{x_3\}, X\})$ vers l'espace de même ensemble de points $X$ et de structure 
\[\mathcal{L}=\{\emptyset,\{x_1\}, \{x_2\}, \{x_3\}, \{x_1, x_2\}, \{x_2, x_3\},  X\}\] défini par  $f(x_i)=x_i$ pour tout $i$, car il n'existe pas d'application croissante non triviale $f:\mathcal{G}_{(X,\mathcal{K})}\rightarrow \mathcal{G}_{(X,\mathcal{L})}$, et en particulier il n'en n'existe par telle que $f\{x_i\}=\{x_i\}$.
\end{rmq}

\subsection{Application $\mathcal{Z}$ telle que $(\mathcal{A},\leq)\sim_{\mathrm{M}} \mathcal{Z}_{(\mathcal{A},\leq)}\in \vert \mathbf{FCnc}\vert$}\label{subs equ Z}

Étant donné $(\mathcal{A},\leq)$ un ensemble fini (partiellement) ordonné quelconque, notons $\mathcal{Z}=\mathcal{Z}_{(\mathcal{A},\leq)}$ l'espace connectif $(Z, \mathcal{L})$ défini de la façon suivante :
\begin{itemize}
\item $Z=\mathcal{A}$,
\item $\mathcal{L}\subset \mathcal{P}Z$ est la structure connective sur $Z$ engendrée par les parties de $Z$ de la forme $\downarrow z=\{x\in Z, x\leq_{\mathcal{A}} z\}$ :
\[\mathcal{L}=[\{\downarrow z, z\in Z\}].\]
\end{itemize}

Il est immédiat que, pour tout $z\in Z$, on a $\downarrow z\in \mathrm{Irr}_\mathcal{Z}$, puisque si $\downarrow z$ n'était pas irréductible il devrait y avoir une partie propre $\downarrow y\subsetneqq \downarrow z$ telle que $z\in\downarrow y$, ce qui est absurde. On en déduit que $(\mathrm{Irr}_\mathcal{Z},\subset)=(\{\downarrow z, z\in Z\}, \subset)$. D'après la proposition \ref{prop equi FCnc-FPos}, on a 
\[
\mathcal{Z}\sim_{\mathrm{M}}(\{\downarrow z, z\in Z\}, \subset).
\]
Mais l'application $z\mapsto\downarrow z$ constitue un isomorphisme d'ensembles ordonnés $(\mathcal{A},\leq)\simeq(\{\downarrow z, z\in Z\}, \subset)$, d'où l'on déduit la proposition suivante :

\begin{prop}\label{prop equi FPos-FCnc}
Pour tout ensemble ordonné $(\mathcal{A},\leq)$, on a 
\[(\mathcal{A},\leq)\sim_{\mathrm{M}}\mathcal{Z}_{(\mathcal{A},\leq)}.\]
\end{prop}

\begin{rmq} L'espace connectif $(Z, \mathcal{L})=\mathcal{Z}=\mathcal{Z}_{(\mathcal{A},\leq)}$, dont la construction s'appuie largement sur l'existence de points non connexes puisque seuls les éléments minimaux de $(\mathcal{A},\leq)$ donnent lieu à des points connexes, n'est pas en général l'espace connectif le plus \og sobre\fg\, qui soit Morita-équivalent à l'ensemble ordonné $(\mathcal{A},\leq)$. Ainsi, si $(X,\mathcal{K})$ désigne l'espace connectif formé de deux points connexes connectés (voir l'exemple \ref{exm deux points connexes connectes}), son graphe générique $\mathcal{G}$ donne lieu à un espace connectif $(Z,\mathcal{L})=\mathcal{Z}_\mathcal{G}$ qui possède un point supplémentaire non connexe, de sorte que $(X,\mathcal{K})$ est plus \og économique\fg\, en points (ou \og sobre\fg) que $(Z,\mathcal{L})$.
\end{rmq}

\begin{rmq}[$\mathcal{Z}$ n'est pas fonctoriel]\label{rmq Z non fonctoriel}

Tout comme $\mathcal{G}$, l'application $\mathcal{Z}$ ne semble pas devoir donner lieu à un foncteur $\mathcal{Z}:\mathbf{FPos}\rightarrow \mathbf{FCnc}$. On ne voit pas, par exemple, quel morphisme connectif associer à l'application croissante $f:\mathcal{A}_1=(\{x_1\}, =)\rightarrow (\{x_2,y_2\}, x_2\leq y_2)=\mathcal{A}_2$ définie par $f(x_1)=y_2$, le point $y_2\in \mathcal{Z}_{\mathcal{A}_2}$ auquel il serait naturel d'associer $x_1\in \mathcal{Z}_{\mathcal{A}_1}$ n'étant pas connexe.
\end{rmq}

\subsection{Foncteur $\mathcal{H}:\mathbf{FTop}\rightarrow \mathbf{FPos}$ tel que $(E,\mathcal{T})\sim_{\mathrm{M}}\mathcal{H}_{(E,\mathcal{T})}$}\label{subs fonct H}

Soit $(E,\mathcal{T})$ un espace topologique fini. Nous dirons qu'un ouvert $U\in\mathcal{T}$ est irréductible s'il n'est pas l'\emph{union} d'une famille de ses parties ouvertes propres, autrement dit s'il est non vide et qu'il n'est pas l'union de deux parties ouvertes propres de $U$. Autrement dit encore, notant $(\mathcal{H}, \leq)$ l'ensemble ordonné par l'inclusion des parties ouvertes irréductibles de $(E,\mathcal{T})$, on a, pour tout ouvert $U$ \emph{non vide},
\[
U\in \mathcal{H} \Leftrightarrow 
\left(\forall (V,W)\in \mathcal{T}^2, U=V\cup W \Rightarrow (V=U \,\mathrm{ou}\, W=U)\right).
\]

Un singleton ouvert étant nécessairement irréductible, on vérifie im\-média\-tement par  une récurrence finie évidente  que tout ouvert de $E$ est l'union des ouverts ir\-ré\-ducti\-bles qu'il contient. Il en découle que l'ensemble ordonné $(\mathcal{H}, \leq)$ est une sous-catégorie de $(\mathcal{T},\subset)$ dense pour la topologie de Grothendieck canonique $J_{\mathcal{T}}$, de sorte que le lemme de comparaison de Grothendieck s'applique. On en déduit l'équivalence de Morita annoncée :

\begin{prop}\label{prop equi FTop-FPos}
Pour tout espace topologique fini $(E, mathcal{T})$, on a 
\[(E, mathcal{T})\sim_{\mathrm{M}}\mathcal{H}_{(E,\mathcal{T})}.\]
\end{prop}

\begin{rmq}\label{rmq H est fonctoriel} 
Pour toute partie $B\subset E$ d'un espace topologique fini, notons $\widetilde{B}=\bigcap\{W\in\mathcal{T}, W\supset B\}$ le plus petit ouvert contenant $B$. On a alors le lemme suivant.
\begin{lm} Étant donnée $f:(E_1,\mathcal{T}_1)\rightarrow(E_2,\mathcal{T}_2)$ une application continue entre espaces topologiques finis, et $\omega\in\mathcal{H}_1$ un ouvert irréductible dans $E_1$, on a
\[
\widetilde{f(\omega)}\in\mathcal{H}_2.
\]
\end{lm}
\begin{proof}
Posons $A=\widetilde{f(\omega)}$, et considérons deux ouverts $V$ et $W$ de $E_2$ tels que $A=V\cup W$. On a $\omega\subset f^{-1}(A)=f^{-1}(V)\cup f^{-1}(W)$, de sorte que, $f$ étant continue et $\omega$ étant irréductible, $\omega\subset f^{-1}(V)$ ou $\omega\subset f^{-1}(W)$, d'où l'on déduit $A=V$ ou $A=W$, ce qui prouve que $A$ est irréductible dans $E_2$.
\end{proof}
Le lemme ci-dessus permet alors de prolonger l'application $(E,\mathcal{T})\mapsto \mathcal{H}_{(E,\mathcal{T})}$ aux applications continues entre espaces topologiques finis, en posant pour tout $\omega\in\mathcal{H}_1$
\[
\mathcal{H}(f)(\omega)=\widetilde{f(\omega)}.
\]

Nous laissons au lecteur le soin de vérifier que l'on définit bien ainsi un foncteur covariant $\mathcal{H}:\mathbf{FTop}\rightarrow \mathbf{FPos}$, ce qui peut du reste se déduire également de la dualité d'Alexandrov (voir ci-après la remarque \ref{rmq H par Alexandrov}).

Par contre, $\mathcal{H}$ ne peut pas se prolonger en un foncteur contravariant $\mathbf{FTop}^{op}\rightarrow \mathbf{FPos}$, puisque par exemple l'injection canonique de l'espace vide dans l'espace topologique à un point ne saurait donner lieu à une application croissante de l'ensemble ordonné à un élément vers l'ensemble vide.
\end{rmq}

\begin{rmq}\label{rmq H par Alexandrov} On peut également vérifier que l'ensemble ordonné $\mathcal{H}_{(E,\mathcal{T})}$ est isomorphe à l'ensemble ordonné opposé à celui obtenu en quotientant par la relation d'équivalence $(x\preceq y$ et $y\preceq x)$ le pré-ordre de spécialisation $\preceq$ défini sur l'ensemble $E$ par : $x\preceq y$ si et seulement si $x$ appartient à l'adhérence de $y$. Plus précisément, tout élément $x\in E$ représente l'ouvert irréductible $\widetilde{\{x\}}$, tout ouvert irréductible est de cette forme, et l'on a $x\preceq y$ si et seulement si $\widetilde{\{y\}}\subset \widetilde{\{x\}}$. En particulier, l'application croissante $\mathcal{H}(f)$ que, dans la remarque \ref{rmq H est fonctoriel}, nous avons associée à une application continue quelconque $f$  correspond, dans le cas particulier des espaces topologiques finis et \emph{modulo} le passage aux ordres opposés et aux quotients, à la définition classique du foncteur qui à tout espace d'Alexandrov --- c'est-à-dire tout espace dans lequel l'intersection d'une famille quelconque d'ouverts est encore un ouvert --- associe l'ensemble de ses points pré-ordonné par spécialisation\footnote{Sur la dualité d'Alexandrov, voir la section \textbf{§\,4.1} de \cite{Caramello:20110317}.}.
\end{rmq}

\begin{rmq} L'ensemble des ouverts d'un espace topologique vérifie \emph{aussi}, trivialement, les axiomes des espaces connectifs. Il donc toujours possible d'associer à tout espace topologique l'espace connectif ayant même points et dont les connexes soient les ouverts du premier. Dans ce cas, la structure topologique et la structure connective considérées sont égales \emph{en tant qu'ensembles ordonnés},  mais comme  elles sont en général distinctes \emph{en tant que sites} (les cribles couvrants ne sont pas les mêmes),  les topos associés ne seront pas équivalents en général, comme l'illustre l'exemple élémentaire de l'espace topologique constitué de deux points ouverts et d'un point fermé.
\end{rmq}

\subsection{Foncteur $\mathcal{E}:\mathbf{FPos}\rightarrow \mathbf{FTop}$ tel que $(\mathcal{A},\leq)\sim_{\mathrm{M}}\mathcal{E}_{(\mathcal{A},\leq)}$}

Étant donné $(\mathcal{A},\leq)$ un ensemble fini (partiellement) ordonné quelconque, notons $\mathcal{E}=\mathcal{E}_{(\mathcal{A},\leq)}$ l'espace topologique $(E, \mathcal{T})$ défini de la façon suivante :
\begin{itemize}
\item $E=\mathcal{A}$,
\item $\mathcal{T}\subset \mathcal{P}Z$ est la structure topologique sur $E$ engendrée par les parties de $E$ de la forme $\downarrow e=\{x\in E, x\leq_{\mathcal{A}} e\}$.
\end{itemize}

La structure topologique engendrée par les $\downarrow e$ étant les unions des intersections de telles parties, on en déduit
\begin{itemize}
\item premièrement que tout ouvert est une section commençante pour la relation $\leq_\mathcal{A}$, autrement dit que pour tout ouvert $U\in\mathcal{T}$, tout $x\in U$ et tout $y\in E$, on a $y\leq x \Rightarrow y\in U$, 
\item deuxièmement que les ouverts de $E$ sont en fait exactement les sections commençantes, une telle section $A$ vérifiant $A=\bigcup_{a\in A}{\downarrow a}$.
\end{itemize}

Autrement dit, pour tout $U\in\mathcal{P}E$, on a $U\in\mathcal{T}\Leftrightarrow 
\left( \forall x\in U, {\downarrow x}\subset U\right)$.
 
Il en découle que, pour tout $e\in E$, l'ouvert non vide $\downarrow e$ est irréductible\footnote{Au sens donné en section \textbf{§\,\ref{subs fonct H}} de l'irréductibilité d'un ouvert d'un espace topologique.}, puisque si $\downarrow e$ n'était pas irréductible il devrait y avoir un ouvert propre $V\subsetneqq \downarrow e$ tel que $e\in V$, mais  dans ce cas on aurait $\downarrow e\subset V$, ce qui est absurde.

Par conséquent, $(\mathcal{H}_\mathcal{E},\subset)=(\{\downarrow e, e\in E\}, \subset)$. D'après la proposition \ref{prop equi FTop-FPos}, on a donc 
\[(\{\downarrow e, e\in E\}, \subset)\sim_{\mathrm{M}}\mathcal{E}_{(\mathcal{A},\leq)}.\]
Mais l'application $e\mapsto\downarrow e$ définit un isomorphisme d'ensembles ordonnés $(\mathcal{A},\leq)\simeq(\{\downarrow e, e\in E\}, \subset)$, de sorte que l'on a prouvé la proposition suivante :

\begin{prop}\label{prop equi FPos-FTop}
Pour tout ensemble ordonné $(\mathcal{A},\leq)$, on a 
\[(\mathcal{A},\leq)\sim_{\mathrm{M}}\mathcal{E}_{(\mathcal{A},\leq)}.\]
\end{prop}

\begin{rmq}\label{rmq espace sobre}
Soient $x$ et $y$ deux points de $E=\mathcal{A}$. Supposons que $y$ appartienne à l'adhérence de $x$. Alors $x$ appartient à tout ouvert contenant $y$, et en particulier $x\in \downarrow y$, autrement dit $x\leq y$. On en déduit que deux points ayant la même adhérence sont nécessairement égaux. Par ailleurs, l'espace étant fini, tout fermé irréductible est nécessairement l'adhérence de l'un de ses points. On en déduit que l'espace topologique $\mathcal{E}_{(\mathcal{A},\leq)}$ est sobre. 
\end{rmq}
  
\begin{rmq}\label{rmq E est  fonctoriel}
De même que le foncteur $\mathcal{H}$ défini précédemment, l'application $\mathcal{E}$ se prolonge en un foncteur (covariant) $\mathbf{FPos}\rightarrow\mathbf{FTop}$. Cela peut se vérifier directement : étant donnée $f:(\mathcal{A}_1,\leq_1)\rightarrow(\mathcal{A}_2,\leq_2)$ une application croissante, et $(E_i,\mathcal{T}_i)=\mathcal{E}_{(\mathcal{A}_i,\leq_i)}$ l'espace topologique associé pour $i\in\{1,2\}$ à $\mathcal{A}_i$, l'application $f$ considérée comme application $E_1\rightarrow E_2$ est continue, puisque pour tout élément $y_1\in f^{-1}(\downarrow x_2)$  de l'image réciproque d'un ouvert irréductible $\downarrow x_2\in\mathcal{T}_2$ et tout $x_1\in E_1$, on a l'implication $x_1\leq_1 y_1\Rightarrow f(x_1)\leq_2 f(y_1)\leq_2 x_2$, d'où $x_1\in f^{-1}(\downarrow x_2)$, de sorte que $\downarrow y_1 \subset f^{-1}(\downarrow x_2)$, ce qui revient à dire que $f^{-1}(\downarrow x_2)$ est ouvert, d'où l'on déduit que l'image réciproque par $f$ de tout ouvert de $E_2$ est ouvert dans $E_1$. La fonctorialité de $\mathcal{E}$ est alors évidente. D'un autre côté, de même que pour le foncteur $\mathcal{H}$ considéré en section \textbf{§\,\ref{subs fonct H}}, la définition et les propriétés du foncteur $\mathcal{E}$ peuvent être ramenées à la dualité d'Alexandrov entre les espaces d'Alexandrov et les ensembles pré-ordonnés. Plus précisément, notant  $TA$ le foncteur qui à tout ensemble pré-ordonné associe l'espace topologique d'Alexandrov constitué des mêmes points et dont les ouverts sont les sections finissantes, le foncteur $\mathcal{E}$ consiste à appliquer successivement le foncteur d'inversion de l'ordre et le foncteur $TA$ aux ensembles ordonnés finis. De ces diverses propriétés, on déduit notamment que $\mathcal{H}\circ \mathcal{E}\simeq Id_{\mathbf{FPos}}$ et que $\mathcal{E}$ est adjoint à gauche de $\mathcal{H}$. Remarquons enfin que pour tout espace topologique fini $(E,\mathcal{T})$, l'espace $(\mathcal{E}\circ \mathcal{H})(E,\mathcal{T})$ est l'unique espace topologique sobre Morita-équivalent à $(E,\mathcal{T})$.
\end{rmq}

\subsection{Exemple et conclusion}  

Des propriétés des applications $\mathcal{G}$ et $\mathcal{Z}$ et des foncteurs $\mathcal{H}$ et $\mathcal{E}$, on déduit notamment que tout espace connectif fini est Morita-équivalent à un espace topologique fini sobre --- d'où la proposition \ref{prop tout connectif fini est Morita topologique} découle --- et que tout espace topologique fini est Morita-équivalent à un espace connectif fini. 

\begin{exm}
L'espace borroméen à trois points (voir l'exemple \ref{exm borro}) est Morita-équivalent à l'espace topologique à \emph{quatre} points dont trois points ouverts qui en engendrent la structure.
\end{exm}

Comme signalé, les applications $\mathcal{G}$ et $\mathcal{Z}$ ne sont pas fonctorielles. Aussi, toujours dans le cas fini, l'étude de la géométricité des morphismes connectifs reste-t-elle à mener : quels sont, parmi les morphismes connectifs, ceux qui \og passent aux topos\fg\,? Une telle étude s'intéressera en particulier aux morphismes connectifs \og irréductibles\fg, à savoir ceux qui transforment les connexes irréductibles en connexes irréductibles. Quant à la question des équivalences de Morita entre espaces topologiques et espaces connectifs dans le cas général, elle reste entièrement ouverte.

\section{Appendices}

\subsection{Une autre preuve de la proposition \ref{prop tout connectif fini est Morita topologique}}\label{append autre preuve}

La proposition \ref{prop tout connectif fini est Morita topologique} découle des propositions \ref{prop equi FCnc-FPos} et \ref{prop equi FPos-FTop} ainsi que de la remarque \ref{rmq espace sobre}. Comme annoncé, en voici une preuve courte, qui utilise des résultats connus mais qui présente l'inconvénient de ne pas expliciter la construction des espaces topologiques ainsi associés aux espaces connectifs.

\begin{proof}
Soit $(X,\mathcal{K})$ un espace connectif fini, de site $(\mathcal{K},J)$ et de topos $\mathcal{E}= Sh(\mathcal{K},J)$.
La catégorie $\mathcal{K}$ étant un ensemble ordonné, le topos $\mathcal{E}$ est localique d'après le théorème 1 de 
\textbf{§\,IX.5} 
donné dans \cite{MacLaneMoerdijk:1992}, 
et il est engendré par les sous-objets de son objet terminal $1_\mathcal{E}$. On en déduit qu'il existe un locale fini $L$ tel que $\mathcal{E}\simeq Sh(L)$. Or, d'après le  \emph{théorème de représentation de Birkhoff} ---
qui affirme qui tout treillis fini distributif $L$ est isomorphe au treillis $(\{\downarrow x, x\in \mathrm{Irr}(L)\},\subset)$ des parties de l'ensemble $\mathrm{Irr}(L)$ des éléments irréductibles de $L$ qui sont de la forme $\downarrow x=\{a \in Irr(L), a \leq x\}$, ensemble de parties qui est évidemment une topologie sur $\mathrm{Irr}(L)$ ---
tout locale fini est spatial. Il existe ainsi, à homéo\-mor\-phisme près, un unique espace topologique fini sobre  $T$ tel que $L\simeq\mathcal{O}(T)$, d'où $Sh(L) \simeq Sh(T)$ et finalement 
\[
\mathcal{E}\simeq Sh(T),
\] 
l'unicité de $T$ étant assurée par l'hypothèse de sa sobriété\footnote{
En effet, d'une part si deux locales ont le même topos, ces locales sont homéomorphes d'après la proposition 2 de la section \textbf{§\,IX.5} de \cite{MacLaneMoerdijk:1992}, et d'autre part un locale spatial est le locale associé, à homéomorphisme près, à un unique espace sobre, à savoir l'espace des points du locale en question (voir la proposition 2 de la section \textbf{§\,IX.3} de \cite{MacLaneMoerdijk:1992}).
}.
\end{proof}

\subsection{Vérification de $Sh(X,\mathcal{K})\simeq \widehat{\mathrm{Irr}_{(X,\mathcal{K})}}$ en termes de faisceaux}\label{subs appendice verif de G par les faisceaux}

Comme indiqué au début de la section \ref{sec Morita equi}, avant de faire appel au lemme de comparaison de Grothendieck, nous avions dans un premier temps établi les équivalences de Morita considérées par des constructions explicites en termes de faisceaux. A titre indicatif, donnons par exemple la preuve de l'équivalence de Morita entre un espace connectif fini quelconque $(X,\mathcal{K})$ et l'ensemble ordonné de ses connexes irréductibles $(\mathcal{G},\leq)= (\mathrm{Irr}_{(X,\mathcal{K})},\subset)$  en explicitant un foncteur ${R}:Sh(X,\mathcal{K})\rightarrow \widehat{(\mathcal{G},\leq)}$ et un foncteur $S:\widehat{(\mathcal{G},\leq)}\rightarrow Sh(X,\mathcal{K})$ constituant une équivalence entre les topos $Sh(X,\mathcal{K})$ et $\widehat{(\mathcal{G},\leq)}=\mathbf{Ens}^{(\mathcal{G},\geq)}$. 

Définissons d'abord le foncteur ${R}:Sh(X,\mathcal{K})\rightarrow \widehat{(\mathcal{G},\leq)}$ : pour tout faisceau $\varphi\in Sh(X,\mathcal{K})$, on définit $R\varphi$ par simple restriction de $\varphi$ aux connexes irréductibles $K\in \mathcal{G}\subset \mathcal{K}$; autrement dit, pour tout $K\in \mathrm{Irr}_{(X,\mathcal{K})}$, on a $R\varphi(K)=\varphi(K)$ et, pour tout morphisme de faisceaux $(\varphi_1\stackrel{\alpha}{\rightarrow}\varphi_2)\in \overrightarrow{Sh(X,\mathcal{K})}$, 
\[
(R\alpha)_K=\alpha_K : \varphi_1(K)\rightarrow \varphi_2(K).
\]

Le foncteur  $S:\widehat{(\mathcal{G},\leq)}\rightarrow Sh(X,\mathcal{K})$ est alors ainsi défini : pour tout préfaisceau $\psi\in \widehat{(\mathcal{G},\leq)}$, le faisceau $S\psi$ prend pour valeur en un connexe $A\in\mathcal{K}$ le produit fibré des valeurs prises par $\psi$ sur les connexes irréductibles inclus dans $A$. Plus précisément, 
\begin{itemize}
\item si $K\in \mathrm{Irr}_{(X,\mathcal{K})}=\mathcal{G}$,  on prend $S\psi(K)=\psi(K)$, de sorte que \[S\psi(K)\simeq {\lim_{\longleftarrow}}_{\{L\in\mathcal{G}, L\subseteq K\}}\psi, \]
\item si $A\in \mathrm{Red}_{(X,\mathcal{K})}=\mathcal{K}\setminus\mathcal{G}$, on prend \[S\psi(A)= {\lim_{\longleftarrow}}_{\{L\in\mathcal{G}, L\subset A\}}\psi,\]
\end{itemize}
les applications de restriction étant celles naturellement associées à ces produits fibrés, à savoir en particulier dans le cas des inclusions \og immédiates\fg\footnote{Nous disons que l'inclusion entre connexes $A\subset B$ est \emph{immédiate} s'il n'existe pas de connexe $C$ tel que $A\subsetneqq C \subsetneqq B$.} suivantes :
\begin{itemize}
\item si $A\subset K$, avec $A\in \mathrm{Red}_{(X,\mathcal{K})}$ et $K\in \mathrm{Irr}_{(X,\mathcal{K})}$, l'application de restriction $\rho_{S\psi_K, S\psi_A}$ est l'application $\psi_K \dashrightarrow {\lim_{\longleftarrow}}_{\{L\in\mathcal{G}, L\subset A\}}\psi$ définie à partir des restrictions $\rho_{\psi K,\psi L}$ par la propriété universelle du produit fibré ${\lim_{\longleftarrow}}_{\{L\in\mathcal{G}, L\subset A\}}\psi$,
\item si $C\subset A$, avec $C\in\mathcal{K}$ et $A\in \mathrm{Red}_{(X,\mathcal{K})}$, la restriction $\rho_{S\psi_A, S\psi_C}$ est la projection sur les composantes définies par $S\psi_C$  des familles constituant le produit fibré ${\lim_{\longleftarrow}}_{\{L\in\mathcal{G}, L\subset A\}}\psi$,
\item si $K_1\subset K_2$, ces deux connexes étant irréductibles, l'application de restriction $\rho_{S\psi_{K_2}, S\psi_{K_1}}$ est la même que dans $\mathcal{G}$.
\end{itemize}
Cette définition de l'image par $S$ des préfaisceaux $\psi\in\widehat{\mathcal{G}}$ se prolonge sans difficulté aux morphismes de  préfaisceaux $\psi_1\stackrel{\beta}{\rightarrow}\psi_2$, l'image, pour tout connexe $A\in\mathcal{K}$, d'une famille cohérente $(k_i,...)\in S\psi_1(A)$ par  $S\beta_A$ étant donnée par la famille cohérente $(\beta_{K_i}(k_i),...)\in S\psi_2(A)$. 

Par construction, on a alors trivialement $R\circ S=id_{\widehat{\mathcal{G}}}$, tandis que l'isomorphisme naturel $S\circ R \simeq Id_{Sh(X,\mathcal{K})}$ découle, quant à lui, des propriétés d'associativité des produits fibrés : intuitivement, la pleine fidélité du foncteur $R$ découle du fait que les irréductibles engendrent la structure connective, de sorte qu'à n'oublier que les $\varphi_A$, où $A$ décrit l'ensemble des connexes réductibles, on n'oublie rien du tout puisque les $\varphi_K$, avec $K\in\mathrm{Irr}_{(X,\mathcal{K})}$, permettent de les retrouver à isomorphisme près. Plus précisément, pour tout connexe $A\in\mathcal{K}$, désignant par $\sigma_{\{L\in\mathcal{G}, L\subset A\}}$ le crible sur $A$ engendré par les irréductibles inclus dans $A$, on a par définition d'un faisceau $\varphi$ :
\[
\varphi_A\simeq {\lim_{\longleftarrow}}_{\sigma_{\{L\in\mathcal{G}, L\subset A\}}}\varphi,
\]
d'où l'on déduit, en retirant par récurrence finie celles des composantes des familles cohérentes d'amalgamation $(k_i, b_j,...)$ associées aux connexes réductibles $B\subset A$, 
\[
\varphi_A\simeq {\lim_{\longleftarrow}}_{{\{L\in\mathcal{G}, L\subset A\}}}\varphi,
\]
autrement dit
\[
\varphi_A\simeq (S\circ R(\varphi))_A.
\]

\bibliographystyle{plain}




\newpage

 \tableofcontents

\end{document}